\documentclass[english]{article}
\usepackage[T1]{fontenc}
\usepackage[latin9]{inputenc}
\usepackage{mathrsfs}
\usepackage{amsmath}
\usepackage{amsthm}
\usepackage{amssymb}
\usepackage{graphicx}

\makeatletter
\numberwithin{equation}{section}
\numberwithin{figure}{section}

\newtheorem{theorem}{Theorem}[section]
\newtheorem{lemma}[theorem]{Lemma}
\newtheorem{corollary}[theorem]{Corollary}

\theoremstyle{definition}

\newtheorem{example}[theorem]{Example}

\theoremstyle{remark}
\newtheorem*{remark}{Remark}

\newcommand\shorttitle{Dimension bounds for invariant measures of bi-Lipschitz IFS}
\newcommand\authors{Andreas Anckar}

\usepackage{fancyhdr}
\@ifundefined{showcaptionsetup}{}{%
 \PassOptionsToPackage{caption=false}{subfig}}
\usepackage{subfig}
\makeatother

\usepackage{babel}
\begin{document}

\title{Dimension bounds for invariant measures of bi-Lipschitz iterated
function systems}

\author{Andreas Anckar\\
Åbo Akademi University, Department of Mathematics\\
Fänriksgatan 3, 20500 Turku, Finland}

\date{1.10.2015}
\maketitle
\begin{abstract}
We study probabilistic iterated function systems (IFS), consisting
of a finite or infinite number of average-contracting bi-Lipschitz
maps on $\mathbb{R}^{d}$. If our strong open set condition is also
satisfied, we show that both upper and lower bounds for the Hausdorff
and packing dimensions of the invariant measure can be found. Both
bounds take on the familiar form of ratio of entropy to the Lyapunov
exponent. Proving these bounds in this setting requires methods which
are quite different from the standard methods used for average-contracting
IFS's.
\end{abstract}

\section{Introduction}

When studying the dimension of measures for probabilistic iterated
function systems, there are two common approaches. One method is to
examine small balls centered at typical points for the Markov chain,
and estimate the logarithmic densities (\ref{eq:mu_dimH_lower_2})-(\ref{eq:mu_dimP_upper_2}).
Since this method requires some intimate understanding of the geometry
of the system, it has proven to be effective primarily in the simple
case of similitudes satisfying the so-called open set condition, which
limits the overlap of the maps. In this scenario, the exact dimensional
value of the invariant measure $\mu$ has been determined in eg. \cite{Deliu1991}
(the finite case) and later generalized in \cite{Moran1997} (the
infinite case). 

Another common approach is to assume very little of the maps and only
search for an upper bound $s$ for the dimension of $\mu$, by explicitly
constructing a set of full $\mu$-measure whose $s$-dimensional Hausdorff
measure is zero. Usually only average contractivity (a notion introduced
in \cite{Freedman1999}) is assumed. This approach has been used in
eg. \cite{Nicol2002}, \cite{Jordan2008} and \cite{Jaroszewska2008}.
While this method is easier to apply to a wider class of maps, the
drawbacks are that it does not give any lower bound for the dimension,
nor does it shed any light on the packing dimension of $\mu$. 

The aim of this paper is to obtain results for a special class of
IFS by merging these two approaches in a way, primarily by extending
techniques used for analyzing strictly contracting IFS's. We will
focus on the case where the maps are bi-Lipschitz. By assuming average
contractivity, the support of $\mu$ is not necessarily bounded and
many of the initial assumptions in the case of similitudes fail. We
will extend some ideas from the second approach to show that we can
still find lower and upper bounds for both the Hausdorff and packing
dimensions of $\mu$.

The motivation for this paper rises from the fact that lower bounds
of the dimension of $\mu$ are not commonly investigated. One example
is \cite{Myjak2002}, where $\mu$ however is required to have bounded
support. Here we will prove the intuitive result that the lower Lipschitz
condition implies a lower bound in a similar way that the upper Lipschitz
condition usually implies an upper bound. Our particular scenario
is interesting because the method we use to prove the dimension bounds
(the lower one in particular) differs from the ``standard'' technique
for similar settings.

We will define a (probabilistic) iterated function system (IFS) as
a set $\mathbb{X}\subset\mathbb{R}^{d}$ associated with a family
of maps $\mathscr{W}=\left\{ w_{i}\right\} _{i\in M}$, $w_{i}:\mathbb{X}\rightarrow\mathbb{X}$,
where the maps are chosen independently according to a probability
vector $\mathbf{p}=\left\{ p_{i}\right\} _{i\in M}$, where $p_{i}>0$
for all $i\in M$ and $\sum_{i\in M}p_{i}=1$. The index set $M$
is either finite or (countably) infinite. We will assume that the
maps are bi-Lipschitz: for every $i\in M$ there exist constants $\Gamma_{i}$
and $\gamma_{i}$, such that $0<\gamma_{i}<\Gamma_{i}$ and 
\begin{eqnarray}
\left|w_{i}(x)-w_{i}(y)\right| & \geq & \gamma_{i}\left|x-y\right|\label{eq:lower_lipschitz}\\
\left|w_{i}(x)-w_{i}(y)\right| & \leq & \Gamma_{i}\left|x-y\right|,\label{eq:upper_lipschitz}
\end{eqnarray}
for all $x,y\in\mathbb{X}$. We will use the notation $\left\{ \mathbb{X},\mathscr{W},\mathbf{p}\right\} $
for an IFS as described above.

Let $M^{\infty}=M\times M\times\ldots$ and define the infinite-fold
product probability measure on $M^{\infty}$ as $\mathbb{P}=\mathbf{p}\times\mathbf{p}\times\ldots$.
We define the mapping $\pi:M^{\infty}\rightarrow\mathbb{R}^{d}$ by
\begin{equation}
\pi\left(i_{1},i_{2},\ldots\right)=\lim_{n\rightarrow\infty}w_{i_{1}}\circ\cdots\circ w_{i_{n}}\left(x_{0}\right)\label{eq:def_pi}
\end{equation}
if the limit exists. 

For an IFS satisfying (\ref{eq:upper_lipschitz}), it is well known
(see \cite{Freedman1999}) that if there exists some $x\in\mathbb{X}$
such that the conditions

\begin{eqnarray}
 & \text{(i)} & \sum_{i\in M}p_{i}\Gamma_{i}<\infty\label{eq:diacfreed1}\\
 & \text{(ii)} & \sum_{i\in M}p_{i}\left|w_{i}\left(x\right)-x\right|<\infty\label{eq:diacfreed2}\\
 & \text{(iii)} & -\infty<\sum_{i\in M}p_{i}\log\Gamma_{i}<0\label{eq:diacfreed3}
\end{eqnarray}
hold, then there exists a unique probability measure $\mu$ on $\mathbb{X}$
such that
\[
\mu=\mathbb{P}\circ\pi^{-1}
\]
on $\mathscr{B}(\mathbb{X})$, the Borel algebra on $\mathbb{X}$.
Alternatively, $\mu$ may be defined as the unique measure satisfying
\[
\mu=\sum_{i\in M}p_{i}\mu\circ w_{i}^{-1}.
\]
The operator $\sum_{i\in M}p_{i}\mu\circ w_{i}^{-1}$ is sometimes
called the Markov\emph{ }operator, and $\mu$ is called the invariant
measure of the IFS. An IFS satisfying conditions (\ref{eq:diacfreed1})-(\ref{eq:diacfreed3})
is said to satisfy average contractivity. In this case there may be
expanding maps in $\mathscr{W}$ and thus $\pi$ does not necessarily
exist for all $\mathbf{i}\in M^{\infty}$. However, the limit (\ref{eq:def_pi})
exists $\mathbb{P}$-a.e. and does not depend on $x_{0}$. 

Another common notion related to iterated function systems is the
open set condition, which is satisfied if there exists an open set
$\mathcal{O}\subset\mathbb{X}$ such that
\begin{eqnarray}
 & \text{(i)} & w_{i}\left(\mathcal{O}\right)\subset\mathcal{O}\text{ for all }i\in M\label{eq:osc1}\\
 & \text{(ii)} & w_{i}\left(\mathcal{O}\right)\cap w_{j}\left(\mathcal{O}\right)=\emptyset\text{ for all }i\neq j.\label{eq:osc2}
\end{eqnarray}
If, in addition, 
\begin{eqnarray}
 & \text{(iii)} & \mathcal{\mu\left(O\right)}>0\label{eq:osc3}
\end{eqnarray}
we will say that the IFS satisfies the strong open set condition (SOSC).
The SOSC is a means to limit the overlapping of the maps, which simplifies
the geometry in ways essential to some of the results below. 

The set of all finite or infinite iterated function systems consisting
of average-contracting bi-Lipschitz maps that satisfy $\sum_{i\in M}p_{i}\log p_{i}>-\infty$
will be denoted $\Xi$. We now state the main result.

\begin{theorem}\label{main}Let $\mu$ denote the invariant measure
of $\left\{ \mathbb{X},\mathscr{W},\mathbf{p}\right\} \in\Xi$. If
the strong open set condition is satisfied, and $\sum_{i\in M}p_{i}\log\gamma_{i}>-\infty$,
then the Hausdorff and packing dimensions of $\mu$ satisfy 
\begin{eqnarray*}
\textup{i)} & \underline{s}\leq\dim_{H}\mu\leq\dim_{H}^{*}\mu\leq\overline{s}\\
\textup{ii)} & \underline{s}\leq\dim_{P}\mu\leq\dim_{P}^{*}\mu\leq\overline{s}
\end{eqnarray*}
where
\[
\underline{s}=\frac{\sum_{M}p_{i}\log p_{i}}{\sum_{M}p_{i}\log\gamma_{i}}\ \text{and\ }\overline{s}=\frac{\sum_{M}p_{i}\log p_{i}}{\sum_{M}p_{i}\log\Gamma_{i}}.
\]
\end{theorem}

\section{Preliminaries}

First, we introduce some notation related to the product space $M^{\infty}=M\times M\times\cdots$,
which consists of infinite sequences of elements from $M$. Here $\mathbb{N}$
will denote the set $\left\{ 1,2,3,\ldots\right\} .$ For any $\mathbf{i}=\left\{ i_{1},i_{2},\ldots\right\} \in M^{\infty}$,
define $\mathbf{i}|k=\left\{ i_{1},i_{2},\ldots,i_{k}\right\} $ and
$\mathbf{i}|j,k=\left\{ i_{j},i_{j+1},\ldots,i_{k}\right\} $ for
$j<k$. For any $k,m\in\mathbb{N}$ and any $\mathbf{i}|k\in M^{k},\,\mathbf{j}|m\in M^{m}$
we will use the operation $\mathbf{i}|k\circ\mathbf{j}|m=i_{1},\ldots,i_{k},j_{1},\ldots,j_{m}$.
By the notation $w_{\mathbf{i}|n}(x)$ we refer to the composition
$w_{i_{1}}\circ w_{i_{2}}\circ\cdots\circ w_{i_{n}}(x)$. In sums
and products we will abbreviate the indexing in the fashion $\sum_{i\in\mathbf{i}|n}p_{i}=\sum_{\mathbf{i}|n}p_{i}$.

We define the overlapping set as 
\[
\Theta=\left\{ x\in\mathbb{X}:\:\exists\mathbf{i},\mathbf{j}\in M^{\infty}\text{ such that }\mathbf{i}\neq\mathbf{j}\text{ and }x=\pi(\mathbf{i})=\pi(\mathbf{j})\right\} .
\]
Let $C_{\mathbf{j}}(n)=\left\{ \mathbf{i}\in M^{\infty}:\,\mathbf{i}|n=\mathbf{j}|n\right\} $,
where $n\in\mathbb{N}$ and $\mathbf{j}\in M^{\infty}$. A set of
this type is called a cylinder set. Cylinder sets generate a natural
topology on $M^{\infty}$ and are clopen sets. For any $\mathbf{i}\in M^{\infty}$
and $A\subset M^{\infty}$, define 
\[
\delta_{\mathbf{i}}(A,n)=\textup{card }\left(A\cap\left\{ \sigma^{k}(\mathbf{i})\right\} _{k=0}^{n-1}\right),
\]
where $\sigma$ denotes the left-shift operator on $M^{\infty}$,
ie. $\sigma\left(\left\{ i_{1},i_{2},\ldots\right\} \right)=\left\{ i_{2},i_{3},\ldots\right\} $.
Let

\[
\mathcal{S}_{0}=\left\{ \mathbf{i}\in M^{\infty}:\,A\in\mathscr{B}\left(M^{\infty}\right)\Rightarrow\lim_{n\rightarrow\infty}\frac{1}{n}\delta_{\mathbf{i}}(A,n)=\mathbb{P}(A)\right\} ,
\]
where $\mathscr{B}\left(M^{\infty}\right)$ is the Borel algebra on
$M^{\infty}$(using the aforementioned cylinder topology). Applying
Birkhoff's ergodic theorem to the probability space $\left\{ M^{\infty},\mathscr{B}\left(M^{\infty}\right),\mathbb{P}\right\} $
and the shift map $\sigma$, we have 
\begin{equation}
\mathbb{P}\left(\mathcal{S}_{0}\right)=1.\label{eq:S0full}
\end{equation}
Now, for a sequence of i.i.d random variables $X_{i}$ with expectancy
$\mathbb{E}\left(X_{i}\right)<\infty$, let 
\[
\mathcal{N}\left(X_{i}\right)=\left\{ \mathbf{i}\in M:\,\lim_{n\rightarrow\infty}\frac{1}{n}\sum_{k=1}^{n}X_{i_{k}}=\mathbb{E}\left(X_{i}\right)\right\} .
\]
We may view $X_{i}^{1}=|w_{i}(x)-x|$, $X_{i}^{2}=\log\Gamma_{i}$
and $X_{i}^{3}=\log p_{i}$ as sequences of i.i.d random variables
with distribution given by $\mathbf{p}$. In light of conditions (\ref{eq:diacfreed1})-(\ref{eq:diacfreed3}),
$\mathbb{E}\left(X_{i}^{k}\right)$ exists for $k=1,2,3$ (independently
of $x$, see \cite{Freedman1999}, corollary 5.2). We now define 
\[
\mathcal{S}=\mathcal{S}_{0}\cap\left(\bigcap_{k=1}^{3}\mathcal{N}\left(X_{i}^{k}\right)\right).
\]
By (\ref{eq:S0full}) and the strong law of large numbers, we have
$\mathbb{P}(\mathcal{S})=1$. We will denote 
\[
\mathcal{A}=\pi\left(\mathcal{S}\right).
\]
While known that $\pi$ exists for almost every $\mathbf{i}\in M^{\infty}$,
the following lemma shows specifically that $\pi$ is defined everywhere
on $\mathcal{S}$. 

\begin{lemma}\label{piexist}$\pi$ is well-defined on $\mathcal{S}$.\end{lemma}

\begin{proof}

Choose $\mathbf{i}\in\mbox{\ensuremath{\mathcal{S}}}$ and $x\in\mathbb{X}$
arbitrarily. Denote $\mathbb{E}\left(\left|w_{i}(x)-x\right|\right)=\lambda_{1}$
and $\mathbb{E}\left(\log\Gamma_{i}\right)=\lambda_{2}$. Fix $c_{1}>\lambda_{1}$,
$\lambda_{2}<c_{2}<0$ and let 
\[
K=c_{1}\sum_{k=0}^{\infty}(k+2)e^{kc_{2}}.
\]
Fix $\epsilon>0$ such that $\epsilon<\min\left\{ c_{i}-\lambda_{i}\right\} _{i=1}^{2}$.
We can choose $N\geq1$ large enough that the conditions
\begin{eqnarray}
\left|w_{i_{n}}(x)-x\right|\leq\sum_{k=1}^{n}\left|w_{i_{k}}(x)-x\right| & \leq & n\left(\lambda_{1}+\epsilon\right)\\
\prod_{k=1}^{n}\Gamma_{i_{k}}=e^{\sum_{k=1}^{n}\log\Gamma_{i_{k}}} & \leq & e^{n\left(\lambda_{2}+\epsilon\right)}\\
ne^{nc_{2}}K & < & \epsilon
\end{eqnarray}
all hold for all $n\geq N$. Note that, by the triangle inequality,
we have
\begin{eqnarray*}
\left|w_{\mathbf{i}|n,m}(x)-x\right| & \leq & \left|w_{i_{n}}(x)-x\right|+\Gamma_{i_{n}}\left|w_{i_{n+1}}(x)-x\right|+\\
 &  & \ldots+\Gamma_{i_{n}}\cdots\Gamma_{i_{m-1}}\left|w_{i_{m}}(x)-x\right|
\end{eqnarray*}
for any $\mathbf{i}\in M^{\infty}$ and $m\geq n$. Now consider,
for $m\geq n>N$,
\begin{eqnarray*}
\left|w_{\mathbf{i}|m}(x)-w_{\mathbf{i}|n}(x)\right| & \leq & \prod_{k=1}^{n}\Gamma_{i_{k}}\left|w_{\mathbf{i}|n+1,m}(x)-x\right|\\
 & \leq & \sum_{j=1}^{m-n}\left(\prod_{k=1}^{n+j-1}\Gamma_{i_{k}}\left|w_{i_{n+j}}(x)-x\right|\right)\\
 & \leq & \sum_{j=1}^{m-n}e^{(n+j-1)\left(\lambda_{2}+\epsilon\right)}\left(n+j\right)\left(\lambda_{1}+\epsilon\right)\\
 & < & e^{nc_{2}}\cdot c_{1}\sum_{j=1}^{\infty}e^{(j-1)c_{2}}(n+j)\\
 & < & ne^{nc_{2}}K\\
 & < & \epsilon.
\end{eqnarray*}
Thus, the sequence $w_{\mathbf{i}|n}(x)$ is Cauchy, meaning that
$\pi(\mathbf{i})=\lim_{n\rightarrow\infty}w_{\mathbf{i}|n}(x)$ exists.
Since $x$ was chosen arbitrarily, the convergence is independent
of $x$.\end{proof}Observe that, for any integers $n\geq m$ and
$\mathbf{i},\mathbf{j}\in M^{\infty}$, $\delta_{\mathbf{i}}\left(C_{\mathbf{j}}(m),n\right)$
denotes the number of times the sequence $\mathbf{j}|m$ appears in
the sequence $\mathbf{i}|m+n-1$. Consequently, for $\mathbf{i}\in\mathcal{S}$
and $\mathbf{j}\in M^{\infty}$, we have
\begin{equation}
\lim_{n\rightarrow\infty}\frac{1}{n}\delta_{\mathbf{i}}\left(C_{\mathbf{j}}(m),n\right)=\mathbb{P}\left(C_{\mathbf{j}}(m)\right)=\prod_{\mathbf{j}|m}p_{i}.\label{eq:lln}
\end{equation}
For $\mathbf{i}\in\mathcal{S}$ and $n\in\mathbb{N}$, define
\[
\tau_{\mathbf{i}}(A,n)=\min\left\{ t:\,\sigma^{t}(\mathbf{i})\in A,\,\delta_{\mathbf{i}}(A,t)=n-1\right\} .
\]
The above expression is well-defined for every $\mathbb{P}$-positive
$A\in\mathscr{B}\left(M^{\infty}\right)$. 

Note that, for any $\mathbf{i}\in M^{\infty}$, $n\in\mathbb{N}$
and $E\subset\mathbb{X}$,
\begin{equation}
\mu\left(w_{\mathbf{i}|n}(E)\right)\geq\mathbb{P}\left\{ \mathbf{j}:\,\mathbf{j}|n=\mathbf{i}|n,\,\pi\left(j_{n+1},j_{n+2},\ldots\right)\in E\right\} =\prod_{\mathbf{i}|n}p_{i}\mu(E).\label{eq:mu_wi_n}
\end{equation}
The inequality in the above equation arises from the fact that we
may have points in $w_{\mathbf{i}|n}(E)$ not necessarily of the form
$\pi\left(i_{1},\ldots i_{n},\ldots\right)$, since there may be several
sequences in $M^{\infty}$ corresponding to the same point in $\mathbb{X}$.
However, in lemma \ref{overlap} it is shown that $\pi$ is injective
on $\mathcal{S}$ if the SOSC holds.

We now give some notation and definitions related to fractal geometry.
For any $\delta>0$, a countable (or finite) collection of sets $\left\{ U_{i}\right\} $
is a $\delta$-cover of a set $E$ if $E\subset\bigcup_{i}U_{i}$
and $\text{diam }U_{i}\leq\delta$ for all $i$. Similarly, a countable
collection of pairwise disjoint balls $\left\{ B_{i}\right\} $, with
centres in $E$ and radii at most $\delta$, is called a \emph{$\delta$}-packing
of $E$. 

For $\delta>0$ and $E\subset\mathbb{R}^{d}$ the $s$-dimensional
Hausdorff measure $\mathcal{H}^{s}$ is defined by 
\[
\mathcal{H}^{s}(E)=\liminf_{\delta\rightarrow0}\left\{ \sum_{i=1}^{\infty}\left|U_{i}\right|^{s}:\,\left\{ U_{i}\right\} \textup{ is a }\delta\textup{-cover of }E\right\} .
\]
The $s$-dimensional packing pre-measure $\mathcal{P}_{0}^{s}$ is
given by
\[
\mathcal{P}_{0}^{s}(E)=\limsup_{\delta\rightarrow0}\left\{ \sum_{i}\left|B_{i}\right|^{s}:\,\left\{ B_{i}\right\} \text{ is a }\delta\text{-packing of }E\right\} ,
\]
whereby the $s$-dimensional packing measure $\mathcal{P}^{s}$ is
defined by
\[
\mathcal{P}^{s}(E)=\inf\left\{ \sum_{i}\mathcal{P}_{0}^{s}\left(E_{i}\right):\,E\subset\bigcup_{i}E_{i}\right\} .
\]
Now we can define the Hausdorff and packing dimensions of $E$ by
\begin{eqnarray*}
\dim_{H}(E) & = & \sup\left\{ s:\,\mathcal{H}^{s}(E)=\infty\right\} =\inf\left\{ s:\,\mathcal{H}^{s}(E)=0\right\} \\
\dim_{P}(E) & = & \sup\left\{ s:\,\mathcal{P}^{s}(E)=\infty\right\} =\inf\left\{ s:\,\mathcal{P}^{s}(E)=0\right\} 
\end{eqnarray*}
We will use two common definitions of the dimension of a measure.
The lower and upper Hausdorff dimensions of a probability measure
$\mu$ are given by $\dim_{H}\mu=\inf\left\{ \dim_{H}(E):\mu(E)>0\right\} $
and $\dim_{H}^{\star}\mu=\inf\left\{ \dim_{H}(E):\mu(E)=1\right\} $,
respectively. Similarly, we define the upper and lower packing dimensions
of $\mu$ to be $\dim_{P}\mu=\inf\left\{ \dim_{P}(E):\mu(E)>0\right\} $
and $\dim_{P}^{\star}\mu=\inf\left\{ \dim_{P}(E):\mu(E)=1\right\} $.
The upper and lower local dimensions of $\mu$ at $x\in\mathbb{R}^{d}$
are given by
\[
\underline{\dim}_{\textup{loc}}\mu(x)=\liminf_{r\rightarrow0}\frac{\log\mu\left(B(x,r)\right)}{\log r}\text{ and }\overline{\dim}_{\textup{loc}}\mu(x)=\limsup_{r\rightarrow0}\frac{\log\mu\left(B(x,r)\right)}{\log r},
\]
where $B(x,r)$ is the closed ball centered around $x$ with radius
$r$. As seen in \cite{Falconer1997}, the above notions of dimension
may be equivalently defined by 
\begin{eqnarray}
\dim_{H}\mu & = & \sup\left\{ s:\,\underline{\dim}_{\textup{loc}}\mu(x)\geq s\text{ for }\mu\text{-almost all }x\right\} \label{eq:mu_dimH_lower_2}\\
\dim_{H}^{*}\mu & = & \inf\left\{ s:\,\underline{\dim}_{\textup{loc}}\mu(x)\leq s\text{ for }\mu\text{-almost all }x\right\} ,\label{eq:mu_dimH_upper_2}
\end{eqnarray}
and
\begin{eqnarray}
\dim_{P}\mu & = & \sup\left\{ s:\,\overline{\dim}_{\textup{loc}}\mu(x)\geq s\text{ for }\mu\text{-almost all }x\right\} \label{eq:mu_dimP_lower_2}\\
\dim_{P}^{*}\mu & = & \inf\left\{ s:\,\overline{\dim}_{\textup{loc}}\mu(x)\leq s\text{ for }\mu\text{-almost all }x\right\} .\label{eq:mu_dimP_upper_2}
\end{eqnarray}

\section{Results}

The following is a generalization of theorem 2.1 in \cite{Moran1998},
and is a key step in proving the lower bound in theorem \ref{main}.

\begin{lemma}\label{overlap}Let $\Theta$ be the overlapping set
for an IFS $\left\{ \mathbb{X},\mathscr{W},\mathbf{p}\right\} \in\Xi$
which also satisfies the SOSC. Then
\[
\mu\left(\Theta\right)=0.
\]
\end{lemma}\begin{proof}Let $\pi_{n}(\mathbf{i)}=\pi\circ\sigma^{n-1}(\mathbf{i})$
and define the orbit of $\mathbf{i}\in\mathcal{S}$ by
\[
O(\mathbf{i})=\left\{ \pi_{n}(\mathbf{i})\right\} _{n=1}^{\infty}.
\]
By lemma \ref{piexist}, the above orbit is well-defined. Also, note
that the symbolic orbit $\left\{ \sigma^{n}\left(\mathbf{i}\right)\right\} _{n=0}^{\infty}$
is dense in $M^{\infty}$.

Suppose there exist $\mathbf{i},\mathbf{j}\in\mathcal{S}$ such that
$\pi(\mathbf{i})=\pi(\mathbf{j})$ and $\mathbf{i}\neq\mathbf{j}$.
Let $k=\min\left\{ n:\,i_{n}\neq j_{n}\right\} $, ie. $k$ is the
first position where $\mathbf{i}$ and \textbf{$\mathbf{j}$} differ.
By (\ref{eq:lower_lipschitz}) the maps $\left\{ w_{i}\right\} _{i\in M}$
are injective, so $\pi_{k}(\mathbf{i})=\pi_{k}(\mathbf{j})$. Furthermore,
$\pi_{k}(\mathbf{i})\in w_{i_{k}}\left(\mathcal{A}\right)$ and $\pi_{k}(\mathbf{j})\in w_{j_{k}}\left(\mathcal{A}\right)$,
so by (\ref{eq:osc2}), at least one of $\pi_{k}(\mathbf{i})\in w_{i_{k}}\left(\mathcal{A}\setminus\mathcal{O}\right)$
and $\pi_{k}(\mathbf{j})\in w_{j_{k}}\left(\mathcal{A}\setminus\mathcal{O}\right)$
must hold. Assuming $\pi_{k}(\mathbf{i})\in w_{i_{k}}\left(\mathcal{A}\setminus\mathcal{O}\right)$,
we have $\pi_{k+1}\left(\mathbf{i}\right)\in\mathcal{A}\setminus\mathcal{O}$,
from which it follows that $\pi_{k+m}(\mathbf{i})\in\mathcal{A}\setminus\mathcal{O}$
for every $m\geq1$, since if $\pi_{k+m}(\mathbf{i})\in\mathcal{O}$
for some $m\geq1$, then by (\ref{eq:osc1}) we have $\pi_{k+1}(\mathbf{i})\in\mathcal{O}$,
which is a contradiction. 

Now, let
\[
E=\left\{ x\in\mathbb{X}:\:\mu\left(B(x,r)\right)>0\;\forall r>0\right\} 
\]
and note that $\mu\left(\mathcal{A}\cap E\right)=1$ since its complement
is clearly a null set. By definition of $\mathcal{S}$, for every
$x\in\mathcal{A}\cap E$ and $\delta>0$ we can find a $k>0$ such
that $\pi_{k}(\mathbf{i})\in B(x,\delta)$, since $\mathbb{P}\left(\pi^{-1}\left(B(x,\delta)\right)\right)>0$.
Thus the orbit $O\left(\mathbf{i}\right)\subset\mathcal{A}\setminus\mathcal{O}$
is dense in $\mathcal{A}\cap E$, which implies that $\left(\mathcal{A}\cap E\right)\cap\mathcal{O}=\emptyset$
and consequently that $\mu\left(\mathcal{O}\right)=0$, contradicting
(\ref{eq:osc3}). Thus $\mathcal{A}\cap\Theta=\emptyset$ and the
proposition follows.\end{proof}An immediate corollary is that if
the SOSC is satisfied, then for any $\mathbf{i}\in M^{\infty}$ and
$n\in\mathbb{N}$ we have 
\begin{equation}
\mu\left(w_{\mathbf{i}|n}\left(\mathcal{A}\right)\right)=\mu\left(w_{\mathbf{i}|n}\left(\overline{\mathcal{O}}\right)\right)=\prod_{\mathbf{i}|n}p_{i}.\label{eq:mu_wi_n_sosc}
\end{equation}
\begin{lemma}\label{tau}Let $\mathbf{i}\in\mathcal{S}$. For any
$A\in\mathscr{B}\left(M^{\infty}\right)$ with $\mathbb{P}(A)>0$
we have 
\[
\lim_{n\rightarrow\infty}\frac{\tau_{\mathbf{i}}(A,n)}{\tau_{\mathbf{i}}(A,n-1)}=1.
\]
\end{lemma}

\begin{proof}Fix $\mathbf{i}\in\mathcal{S}$ and $A\in\mathscr{B}\left(M^{\infty}\right)$
(with $\mathbb{P}(A)>0$). The result follows immediately upon observing
that
\[
\mathbb{P}(A)=\lim_{n\rightarrow\infty}\frac{\delta_{\mathbf{i}}(A,n)}{n}=\lim_{k\rightarrow\infty}\frac{k}{\tau_{\mathbf{i}}(A,k)}.
\]
\end{proof}The proof of theorem \ref{main} is now given by the following
two lemmas.

\begin{lemma}\label{upper}Let $\mu$ be the invariant measure of
$\left\{ \mathbb{X},\mathcal{\mathscr{W}},\mathbf{p}\right\} \in\Xi$.
Then 
\[
\dim_{P}^{*}\mu\leq\frac{\sum_{i}p_{i}\log p_{i}}{\sum_{i}p_{i}\log\Gamma_{i}}.
\]
\end{lemma}\begin{proof}

First note that, for any $\mathbf{i}\in\mathcal{S}$, we have 
\[
\lim_{n\rightarrow\infty}\frac{1}{n}\log\prod_{\mathbf{i}|n}\Gamma_{i}=\lim_{n\rightarrow\infty}\frac{1}{n}\sum_{\mathbf{i}|n}\log\Gamma_{i}=\sum_{M}p_{i}\log\Gamma_{i},
\]
Thus 
\begin{equation}
\lim_{n\rightarrow\infty}\prod_{\mathbf{i}|n}\Gamma_{i}=0,\label{eq:prodG}
\end{equation}
by condition (\ref{eq:diacfreed3}). Now fix $x_{0}\in\mathbb{X}$,
$\epsilon>0$ and choose $R_{\epsilon}$ such that $\mu\left(B\left(x_{0},R_{\epsilon}\right)\right)>1-\epsilon$.
Denote the ball $B\left(x_{0},R_{\epsilon}\right)$ by $B$, and let
$B^{\prime}=\pi^{-1}(B)$. Since $\pi$ is measurable, $B^{\prime}$
is Borel, something which can also be seen by writing

\[
\pi^{-1}\left(B\right)=\bigcap_{K=1}^{\infty}\bigcup_{N=1}^{\infty}\bigcap_{n=N}^{\infty}\left\{ \mathbf{i}\in M^{\infty}:\,\inf_{x\in B}d\left(w_{i_{1}}\circ\cdots\circ w_{i_{n}}\left(x_{0}\right),x\right)<\frac{1}{K}\right\} ,
\]
where $x_{0}\in\mathbb{X}$ is arbitrary. The set within the brackets
above is open (since it is a countable union of cylinder sets) implying
$\pi^{-1}(B)\in\mathscr{B}\left(M^{\infty}\right)$. Now, choose $x=\pi(\mathbf{i})\in B\cap\mathcal{A}$
and fix $r>0$. Let

\[
N=\min\left\{ k:\,\prod_{\mathbf{i}|\tau_{\mathbf{i}}\left(B^{\prime},k\right)}\Gamma_{i}<\frac{r}{2R_{\epsilon}}\leq\prod_{\mathbf{i}|\tau_{\mathbf{i}}\left(B^{\prime},k-1\right)}\Gamma_{i}\right\} .
\]
By (\ref{eq:prodG}), the above constant exists, provided that $r$
is chosen small enough. Note that, by the definition of $\tau\mathbf{_{i}}\left(B^{\prime},k\right)$,
we have $x\in w_{\mathbf{i}|\tau\mathbf{_{i}}\left(B^{\prime},k\right)}(B)$
for all $k\geq1$. Since $\left|w_{\mathbf{i}|n}(B)\right|\leq\prod_{\mathbf{i}|n}\Gamma_{i}2R_{\epsilon}$
(here $\left|\cdot\right|$ denotes the diameter of a set), the set
$w_{\mathbf{i}|\tau_{\mathbf{i}}\left(B^{\prime},N\right)}(B)$ is
certainly included in the ball $B(x,r)$. Combining this with (\ref{eq:mu_wi_n})
yields

\begin{eqnarray*}
\frac{\log\mu\left(B(x,r)\right)}{\log r} & \leq & \frac{\log\mu\left(w_{\mathbf{i}|\tau_{\mathbf{i}}\left(B^{\prime},N\right)}(B)\right)}{\log\left(2R_{\epsilon}\prod_{\mathbf{i}|\tau_{\mathbf{i}}\left(B^{\prime},N-1\right)}\Gamma_{i}\right)}\\
 & \leq & \frac{\log\mu(B)+\log\prod_{\mathbf{i}|\tau_{\mathbf{i}}\left(B^{\prime},N\right)}p_{i}}{\log2R_{\epsilon}+\log\prod_{\mathbf{i}|\tau_{\mathbf{i}}\left(B^{\prime},N-1\right)}\Gamma_{i}}.
\end{eqnarray*}
Applying lemma \ref{tau}, we get

\[
\begin{aligned}\limsup\limits _{r\rightarrow0} & \frac{\log\mu\left(B(x,r)\right)}{\log r}\leq\\
 & \limsup_{N\rightarrow\infty}\frac{\tau_{\mathbf{i}}\left(B^{\prime},N\right)}{\tau_{\mathbf{i}}\left(B^{\prime},N-1\right)}\cdot\frac{\frac{1}{\tau_{\mathbf{i}}\left(B^{\prime},N\right)}\left(\log\mu(B)+\log\prod_{\mathbf{i}|\tau_{\mathbf{i}}\left(B^{\prime},N\right)}p_{i}\right)}{\frac{1}{\tau_{\mathbf{i}}\left(B^{\prime},N-1\right)}\left(\log2R_{\epsilon}+\log\prod_{\mathbf{i}|\tau_{\mathbf{i}}\left(B^{\prime},N-1\right)}\Gamma_{i}\right)}\\
 & =\frac{\sum_{M}p_{i}\log p_{i}}{\sum_{M}p_{i}\log\Gamma_{i}}
\end{aligned}
\]
from which we acquire $\overline{\dim}_{\textup{loc}}\mu\left(x\right)\leq\overline{s}$
for all $x\in\mathcal{A}\cap B$. Now let 
\[
E_{n}=B\left(x_{0},R_{\epsilon/n}\right)\cap\mathcal{A},
\]
where $\left\{ R_{\epsilon/n}\right\} _{n=1}^{\infty}$ is chosen
to be an increasing sequence such that 
\[
\mu\left(B\left(x_{0},R_{\epsilon/n}\right)\right)>1-\frac{\epsilon}{n}.
\]
For every $n\geq1$, the above argument can be repeated so that $\overline{\dim}_{\text{loc}}\mu(x)\leq\underline{s}$
for all $x\in E_{n}$. Since $\mu\left(\bigcup_{n\geq1}E_{n}\right)=\mu\left(\mathcal{A}\right)=1$,
the result follows from (\ref{eq:mu_dimP_upper_2}). \end{proof}

\begin{remark}Note that the SOSC need not be assumed in the above
lemma.\end{remark}

\begin{lemma}\label{lower}Let $\mu$ be the invariant measure of
$\left\{ \mathbb{X},\mathcal{\mathscr{W}},\mathbf{p}\right\} \in\Xi$
and assume that the SOSC holds. If $\sum_{i\in M}p_{i}\log\gamma_{i}>-\infty$,
then
\[
\dim_{H}\mu\geq\frac{\sum_{i}p_{i}\log p_{i}}{\sum_{i}p_{i}\log\gamma_{i}}.
\]
\end{lemma}\begin{proof}Fix $\epsilon>0$ and $x_{0}\in\mathbb{X}$.
As before, we choose $R_{\epsilon}$ such that $\mu\left(B\left(x_{0},R_{\epsilon}\right)\right)>1-\epsilon$,
and denote the ball $B\left(x_{0},R_{\epsilon}\right)$ by $B$. Furthermore,
choose some $\mathbf{j}\in\mathcal{S}$ and integer $K$ such that
$\pi\left(\mathbf{j}\right)\in\mathcal{O}$ and $d\left(w_{\mathbf{j}|K}\left(B\right),\partial\mathcal{O}\right)>\delta$
for some sufficiently small $\delta>0$. Moreover, choose $x=\pi(\mathbf{i})\in B\cap\mathcal{A}$
and note that $\mathbf{i}$ is unique (due to lemma \ref{overlap})
given this $x$. Fix $r>0$ and let
\[
F=C_{\mathbf{j}}(K)\cap\left\{ \mathbf{i}\in\mathcal{S}:\,\pi\left(\sigma^{K}(\mathbf{i})\right)\in B\right\} .
\]
Now define 
\[
N=\max\left\{ k:\,\prod_{\mathbf{i}|\tau_{\mathbf{i}}\left(F,k\right)}\gamma_{i}<\frac{r}{\delta}\leq\prod_{\mathbf{i}|\tau_{\mathbf{i}}\left(F,k-1\right)}\gamma_{i}\right\} .
\]
As before, $N$ exists provided that $r$ is small enough. For any
$n$, the map $w_{\mathbf{i}|n}$ satisfies (\ref{eq:lower_lipschitz})
with lower Lipschitz constant $\prod_{\mathbf{i}|n}\gamma_{i}$, so
\begin{eqnarray*}
B(x,r) & \subset & B\left(x,\delta\prod_{\mathbf{i}|\tau_{\mathbf{i}}(F,N-1)}\gamma_{i}\right)\\
 & \subseteq & w_{\mathbf{i}|\tau_{\mathbf{i}}(F,N-1)}\left[B\left(\pi\circ\sigma^{\tau_{\mathbf{i}}(F,N-1)}(\mathbf{i}),\delta\right)\right]\\
 & = & w_{\mathbf{i}|\tau_{\mathbf{i}}(F,N-1)}\left[B\left(w_{\mathbf{j}|K}\circ\pi\circ\sigma^{\tau_{\mathbf{i}}(F,N-1)+K}(\mathbf{i}),\delta\right)\right]\\
 & \subset & w_{\mathbf{i}|\tau_{\mathbf{i}}(F,N-1)}\left(\mathcal{O}\right)\\
 & \subset & w_{\mathbf{i}|\tau_{\mathbf{i}}(F,N-1)}\left(\overline{\mathcal{O}}\right),
\end{eqnarray*}
since, by definition, we have that $B\left(w_{\mathbf{j}|K}(y),\delta\right)\subset\mathcal{O}$
for any $y\in B$. Consequently,
\begin{eqnarray*}
\mu\left(B(x,r)\right) & \leq & \mu\left(w_{\mathbf{i}|\tau_{\mathbf{i}}\left(F,N-1\right)}\left(\overline{\mathcal{O}}\right)\right)=\prod_{\mathbf{i}|\tau_{\mathbf{i}}\left(F,N-1\right)}p_{i},
\end{eqnarray*}
giving

\begin{eqnarray*}
\frac{\log\mu\left(B(x,r)\right)}{\log r} & \geq & \frac{\log\prod_{\mathbf{i}|\tau_{\mathbf{i}}\left(F,N-1\right)}p_{i}}{\log\delta\prod_{\mathbf{i}|\tau_{\mathbf{i}}\left(F,N\right)}\gamma_{i}}.
\end{eqnarray*}
We now arrive at

\[
\begin{aligned}\liminf_{r\rightarrow0} & \frac{\log\mu\left(B(x,r)\right)}{\log r}\geq\\
 & \liminf_{N\rightarrow\infty}\frac{\tau_{\mathbf{i}}\left(F,N-1\right)}{\tau_{\mathbf{i}}\left(F,N\right)}\cdot\frac{\frac{1}{\tau_{\mathbf{i}}\left(F,N-1\right)}\log\prod_{\mathbf{i}\left(\tau_{\mathbf{i}}\left(F,N-1\right)\right)}p_{i}}{\frac{1}{\tau_{\mathbf{i}}\left(F,N\right)}\left(\log\delta+\log\prod_{\mathbf{i}\left(\tau_{\mathbf{i}}\left(F,N\right)\right)}\gamma_{i}\right)}\\
 & =\frac{\sum_{M}p_{i}\log p_{i}}{\sum_{M}p_{i}\log\gamma_{i}}
\end{aligned}
\]
Again, we applied lemma \ref{tau} since $\mathbb{P}(F)=\prod_{\mathbf{j}|K}p_{i}\cdot\mu(B)>0$.
We have now shown that $\underline{\dim}_{\text{loc}}\mu(x)\geq\underline{s}$
for all $x\in B\cap\mathcal{A}$. Similarly to lemma \ref{upper},
let $E_{n}=B\left(x_{0},R_{\epsilon/n}\right)\cap\mathcal{A}$, and
the result follows by (\ref{eq:mu_dimH_lower_2}).\end{proof}

We may also remark that by standard results (see eg. \cite{Moran1998},
theorem 4.4 in \cite{Moran1997} and proposition 2.2 in \cite{Falconer1997})
the following result holds.

\begin{corollary}Under the conditions in theorem \ref{main} we have
that 
\begin{eqnarray*}
\textup{i)} & \mu\textup{ is absolutely continuous w.r.t. }\mathcal{H}^{d}\mbox{\textup{ and }}\mathcal{P}^{d}\textup{ whenever }d<\underline{s}\\
\textup{ii)} & \mu\textup{ is singular w.r.t. }\mathcal{H}^{d}\mbox{\textup{ and }}\mathcal{P}^{d}\textup{ whenever }d>\overline{s}
\end{eqnarray*}
\end{corollary}

\section{Examples in $\mathbb{R}^{2}$}

The following examples demonstrate cases where all the conditions
of Theorem \ref{main} hold. The maps aren't all strictly contracting,
and the SOSC holds, but the fine structure of $\mu$ can still be
argued to be ``fractal''. The invariant measures are visualized
by plotting the points of the Markov chain (the so-called forward
iteration) 
\[
w_{i_{1}}\left(x_{0}\right),w_{i_{2}}\circ w_{i_{1}}\left(x_{0}\right),\ldots,w_{i_{N}}\circ\cdots\circ w_{i_{1}}\left(x_{0}\right)
\]
for some starting point $x_{0}$, where at each step $i_{n}$ is chosen
(i.i.d) according to the given probability vector.

\begin{example}\label{ex1}Consider an IFS on $\mathbb{X}=[0,\infty)\times[0,1]$
given by
\[
w_{i}(x,y)=\left(\begin{array}{cc}
\Gamma_{i} & 0\\
0 & \gamma_{i}
\end{array}\right)\left(\begin{array}{c}
x\\
y
\end{array}\right)+\left(\begin{array}{c}
i-1\\
\sum_{k=1}^{i-1}\gamma_{k}
\end{array}\right)
\]
where
\[
\gamma_{i}=\frac{2^{i-1}}{3^{i}},\;\Gamma_{i}=\gamma_{i}+\frac{i}{10}
\]
for $i=1,2,\ldots$. We define the probabilities by $p_{i}=2^{-i}$. 

Since the maps are affine it is not hard to see that $\Gamma_{i}$
and $\gamma_{i}$ indeed are the upper and lower Lipschitz constants
for $w_{i}$, respectively. In this particular case the maps are expanding
horizontally and contracting vertically. The maps split $\mathbb{X}$
into smaller and smaller slices, increasingly displaced from the origin
horizontally (see figure \ref{fig:ex1}a). It is also apparent from
the same figure that the SOSC is satisfied with $\mathcal{O}=\left(0,\infty\right)\times\left(0,1\right)$.

Furthermore, we have $\Gamma_{i}>1$ for all $i\geq10$, and $\gamma_{i}<1$
for all $i\geq1$. The IFS belongs to $\Xi$ since $\sum_{i=1}^{\infty}p_{i}\log p_{i}\approx-1.39$
and 
\[
\sum_{i=1}^{\infty}p_{i}\Gamma_{i}\approx0.45,\;\sum_{i=1}^{\infty}p_{i}\log\Gamma_{i}\approx-0.805,\;\sum_{i=1}^{\infty}p_{i}\left|w_{i}\left((0,0)\right)\right|\approx1.03,
\]
whereby the conditions (\ref{eq:diacfreed1})-(\ref{eq:diacfreed3})
are all satisfied. By theorem \ref{main} we now have 
\[
0.92<\dim_{H}\mu\leq\dim_{P}^{*}\mu<1.73.
\]
The invariant measure is visualized in figure \ref{fig:ex1}b.\end{example}
\begin{figure}
\begin{centering}
\caption{Example \ref{ex1}\label{fig:ex1}}

\par\end{centering}

\begin{centering}
\subfloat[Schematic overview of the maps]{\includegraphics[width=0.4\paperwidth]{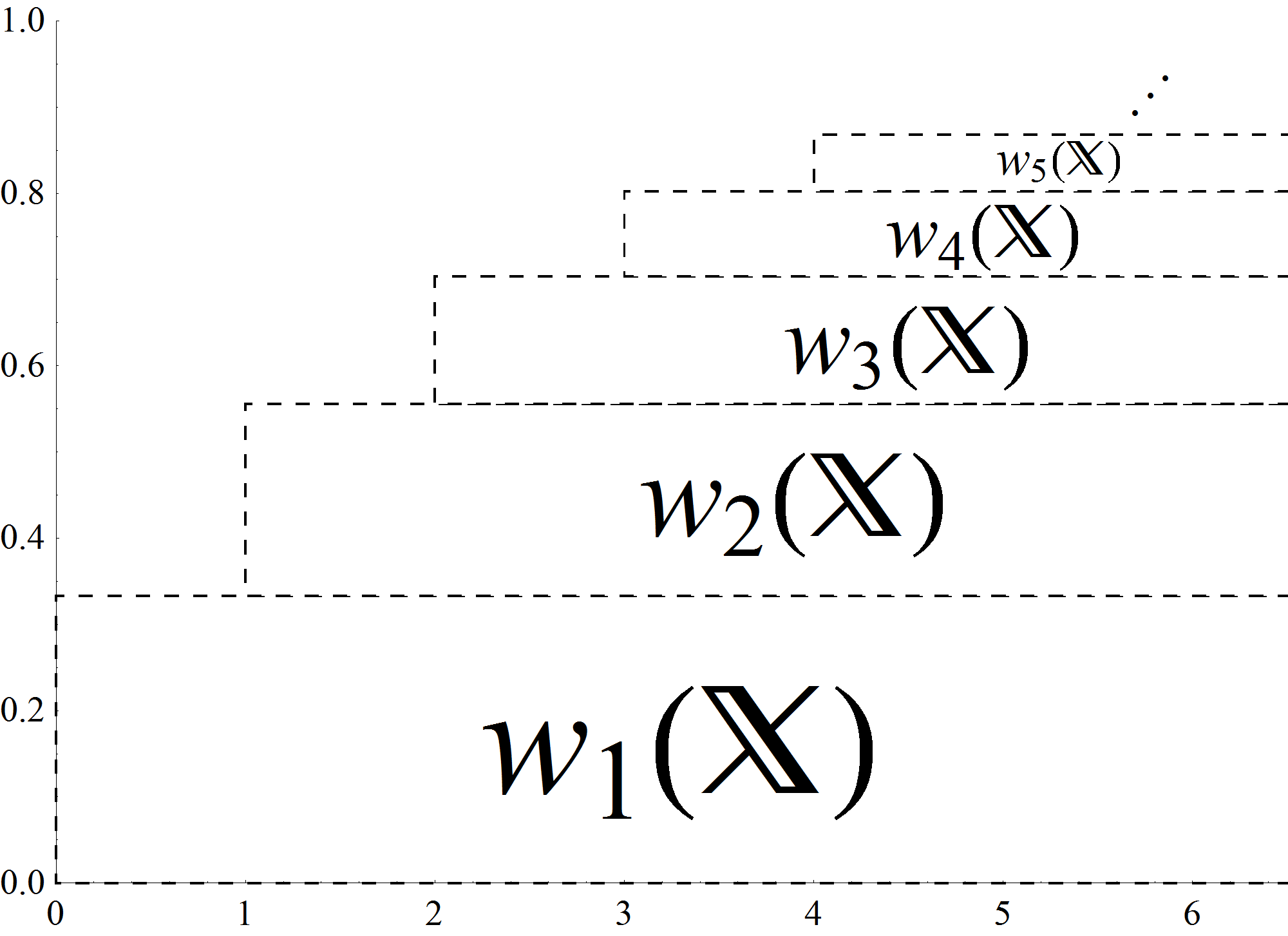}}
\par\end{centering}

\centering{}\subfloat[Plot of the Markov chain]{

\centering{}\includegraphics[width=0.4\paperwidth]{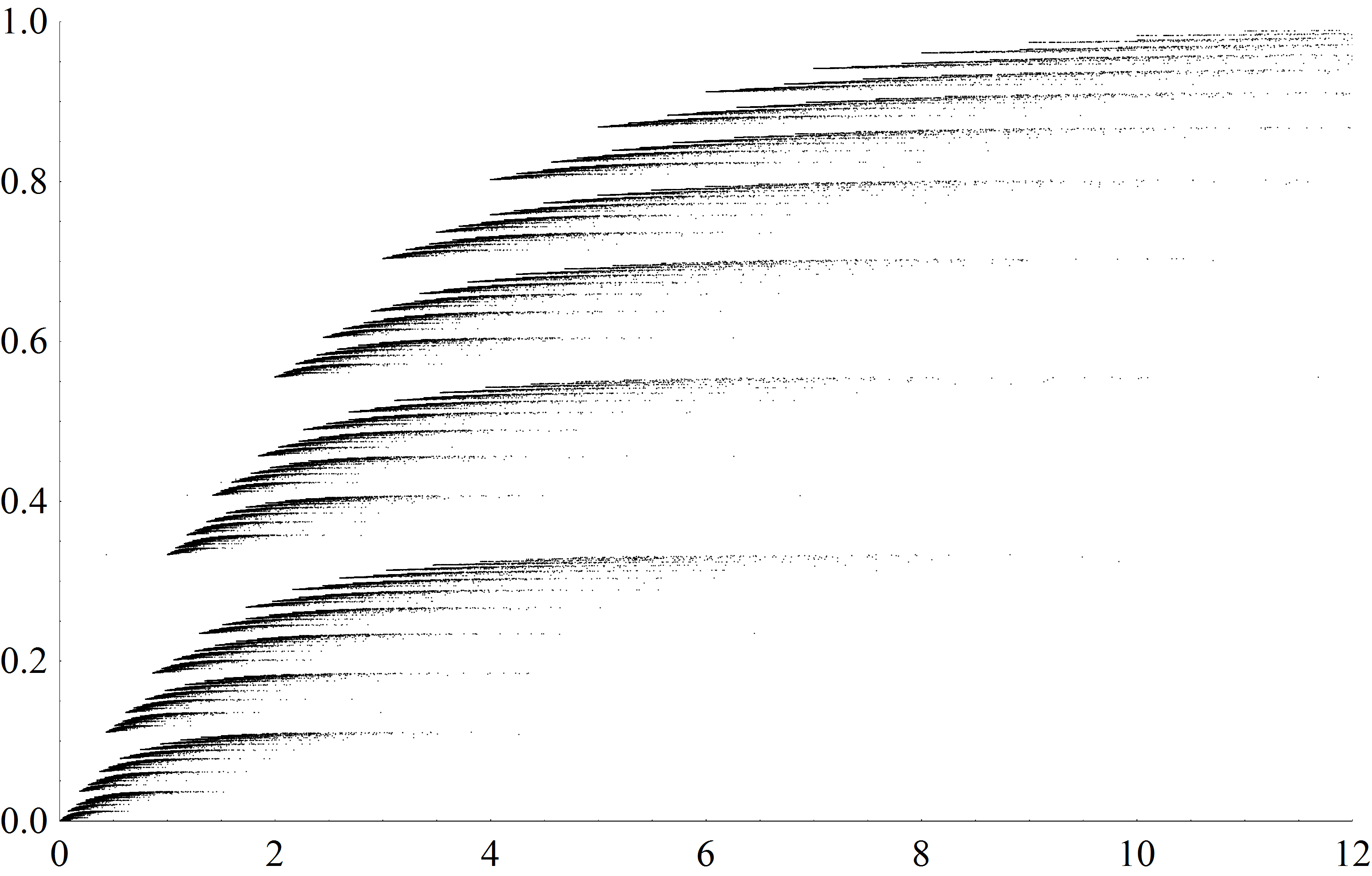}}
\end{figure}
\begin{figure}
\begin{centering}
\caption{Example \ref{ex2}\label{fig:ex2}}

\par\end{centering}

\begin{centering}
\subfloat[Schematic overview of the maps]{\includegraphics[width=0.4\paperwidth]{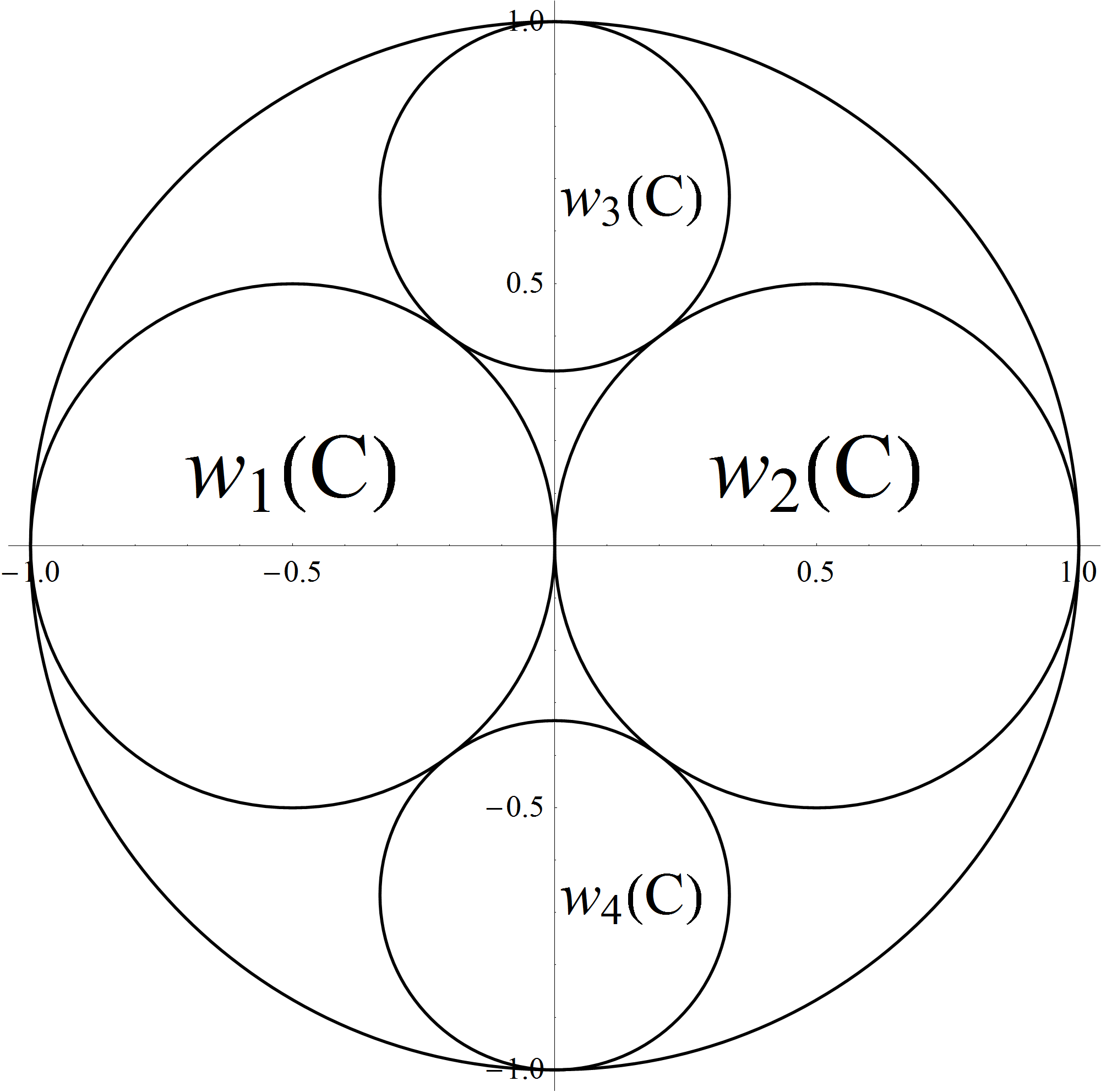}}
\par\end{centering}

\centering{}\subfloat[Plot of the Markov chain]{

\centering{}\includegraphics[width=0.4\paperwidth]{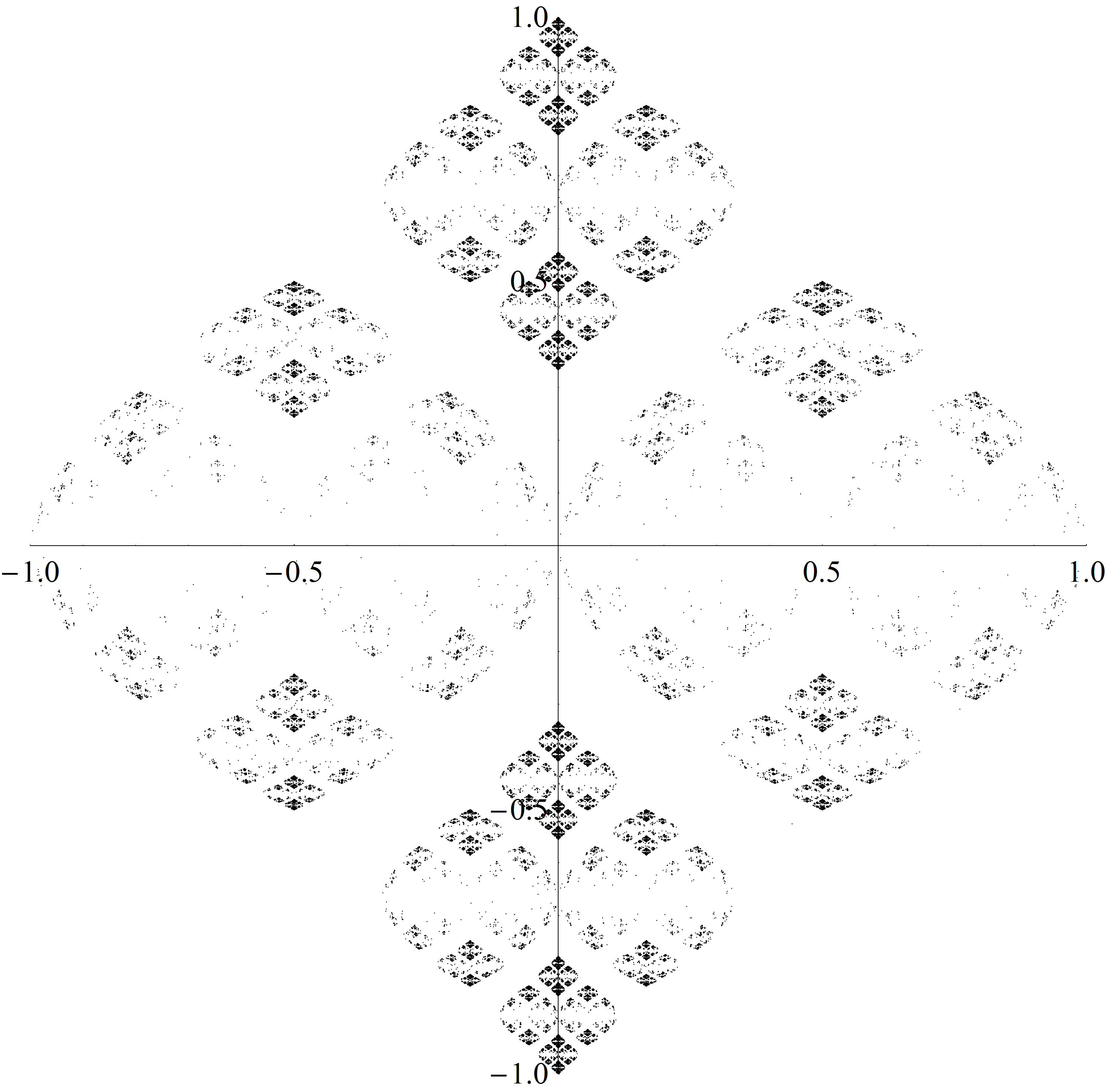}}
\end{figure}

\begin{example}\label{ex2}Define an IFS on the two-dimensional unit
ball $C$, consisting of four maps, by $w_{i}(x,y)=a_{i}\left(\begin{array}{c}
x\\
y
\end{array}\right)+b_{i}$, where 
\[
a_{i}=\begin{cases}
\frac{1}{4}\left(5-3\left(x^{2}+y^{2}\right)^{1/6}\right), & i=1,2\\
\frac{1}{3}, & i=3,4
\end{cases}
\]
and
\[
b_{1}=\left(\begin{array}{c}
-1/2\\
0
\end{array}\right),\,b_{2}=\left(\begin{array}{c}
1/2\\
0
\end{array}\right),\,b_{3}=\left(\begin{array}{c}
0\\
2/3
\end{array}\right).\,b_{4}=\left(\begin{array}{c}
0\\
-2/3
\end{array}\right).
\]
Let $\mathbf{p}=\left\{ 0.1,0.1,0.4,0.4\right\} $. This IFS maps
$C$ into four smaller circles (see the schematic overview in figure
\ref{fig:ex2}). The maps $w_{3}$ and $w_{4}$ are similitudes (ie.
$\Gamma=\gamma=1/3$) while $w_{1}$ and $w_{2}$ expand locally,
around the circle centres, but contract toward the edges. Thus $w_{1}$
and $w_{2}$ are not strict contractions, but the support of the invariant
measure $\mu$ is still bounded and contained in $C$. It can be shown
that $\Gamma_{i}=5/4$ and $\gamma_{i}=1/4$ for $i=1,2$.

The SOSC is satisfied (with $\mathcal{O}$ being the open unit ball),
and we have $\sum p_{i}\log\Gamma_{i}\approx-0.74$ so average contractivity
is fulfilled as well. By theorem \ref{main}, 
\[
1.05<\dim_{H}\mu\leq\dim_{P}^{*}\mu<1.43.
\]
Note that the dimension is strictly smaller than that of the Sierpinski
gasket ($\log3/\log2\approx1.58)$ which would have been achieved
if all the maps were similitudes.\end{example}

\section{Acknowledgements}

I would like to thank Göran Högnäs for comments and suggestions, and
the Finnish Doctoral Programme in Stochastics and Statistics, Academy
of Finland project no. 141318, for financial support.

\bibliographystyle{plain}
\addcontentsline{toc}{section}{\refname}\bibliography{aku}

\end{document}